\documentclass[]{interact}
\usepackage{comment}
\usepackage[numbers,sort&compress]{natbib}
\bibpunct[, ]{[}{]}{,}{n}{,}{,}

\theoremstyle{plain}%
\usepackage[table, svgnames, dvipsnames]{xcolor}
\usepackage{lineno,hyperref}
\modulolinenumbers[5]
\usepackage{natbib}
\usepackage{tabularx}
\usepackage[utf8]{inputenc}
\usepackage{amsthm}
\usepackage{amsmath}
\usepackage{amsfonts}
\usepackage{booktabs}
\newtheorem{theorem}{Theorem}
\newtheorem{corollary}{Corollary}
\usepackage{amssymb}
\usepackage{algorithm}
\usepackage{algcompatible}

\usepackage{graphicx}
\usepackage{xcolor}
\usepackage{xspace}
\usepackage{etoolbox}
\usepackage{appendix}
\usepackage[capitalize]{cleveref}

\AtBeginEnvironment{appendices}{\crefalias{section}{appendix}}

\hypersetup{
    colorlinks,%
    citecolor=green,%
    filecolor=black,%
    linkcolor=red,%
    urlcolor=blue
}

\newcommand{\bmat}[1]{\begin{bmatrix}#1\end{bmatrix}}

\newcommand{\inv}{^{-1}}
\newcommand{\LDLT}{\hbox{LDL$^{\hbox{\scriptsize \raisebox{-1pt}{\!T}}}$}\xspace}
\newcommand{\nnz}{\textrm{nnz}}
\newcommand{\norm}[1]{\|#1\|}

\newcommand{\Hgam}{H_{\gamma}}
\newcommand{\Hkgam}{H_{K\gamma}}
\newcommand{\Hdel}{H_{\delta}}
\newcommand{\Htilde}{\widetilde H}
\newcommand{\Matlab}{{\sc Matlab}\xspace}
\newcommand{\nx}{{n_x}} 

\newcommand{\mc}{{m_c}}
\newcommand{\md}{{m_d}}
\newcommand{\rhat}{\skew3\hat r}
\newcommand{\rtilde}{\tilde{r}}

\newcommand{\xtilde}{\skew3\tilde x}

\newcommand{\Cpp}{C\raisebox{1.5pt}{\small $++$}\xspace}

\newcommand{\fac}{_\textrm{fac}}
\newcommand{\op}{_\textrm{op}}
\newcommand{\Width}{0.9}   

\begin{document}
\title{A Hybrid Direct-Iterative Method for Solving KKT Linear Systems}

\author{
\name{Shaked Regev\textsuperscript{a}\thanks{CONTACT Shaked Regev Email: sregev@stanford.edu},
Nai-Yuan Chiang\textsuperscript{b},
Eric Darve\textsuperscript{a},
Cosmin G. Petra\textsuperscript{b}, \\
Michael A. Saunders\textsuperscript{a},
Kasia \'{S}wirydowicz\textsuperscript{c}, and
Slaven~Pele\v{s}\textsuperscript{c}
}
\affil{\textsuperscript{a}Stanford University, 450 Serra Mall, Stanford,  California, USA\\
\textsuperscript{b}Lawrence Livermore National Laboratory, 7000 East Ave, Livermore, California, USA\\
\textsuperscript{c}Pacific Northwest National Laboratory, 902 Battelle Blvd, Richland, Washington, USA}
}

\maketitle
\begin{abstract}
We propose a solution strategy for linear systems arising in interior method optimization, which is suitable for implementation on hardware accelerators such as graphical processing units (GPUs). The current gold standard for solving these systems is the \LDLT factorization. However, \LDLT requires pivoting during factorization, which substantially increases communication cost and degrades performance on GPUs. Our novel approach solves a large indefinite system by solving multiple smaller positive definite systems, using an iterative solve for the Schur complement and an inner direct solve (via Cholesky factorization) within each iteration. Cholesky is stable without pivoting, thereby reducing communication and allowing reuse of the symbolic factorization. We demonstrate the practicality of our approach and show that on large systems it can efficiently utilize GPUs and outperform \LDLT factorization of the full system.
\end{abstract}


\section{Introduction}
Interior methods for nonlinear optimization~\cite{wachter2006implementation,Byrd2006,Vanderbei1999,GONDZIO1995} are essential in many areas of science and engineering. They are commonly used in model predictive control with applications to robotics \cite{sleiman2019contact}, autonomous cars \cite{wang2020non}, aerial vehicles \cite{jerez2017forces}, combustion engines \cite{keller2020model}, and heating, ventilation and air-conditioning systems \cite{ma2014stochastic}, to name a few. Interior methods are also used in public health policy strategy development \cite{acemoglu2020optimal,silva2013optimal},  data fitting in physics \cite{reinert2018semilocal}, genome science \cite{andronescu2010computational,varoquaux2014statistical}, and many other areas. 

Most of the computational cost within interior methods is in solving linear systems of Karush-Kuhn-Tucker (KKT) type~\cite{NocW2006}. The linear systems are sparse, symmetric indefinite, and usually 
ill-conditioned and difficult to solve. Furthermore, implementations of interior methods for nonlinear optimization,  such as the filter-line-search approach in Ipopt~\cite{wachter2006implementation} and HiOp~\cite{hiopuserguide}, typically expect the linear solver to provide the matrix inertia (number of positive, negative and zero eigenvalues)
to determine if the system should be regularized. (Otherwise, interior methods perform curvature tests to ensure descent in a certain merit function~\cite{Chiang2016}.) Relatively few linear solvers are equipped to solve KKT systems, and even fewer to run those computations on hardware accelerators such as graphic processing units (GPUs)~\cite{swirydowicz2020linear}. 

At the time of writing, six out of the ten most powerful computers in the world have more than 90\% of their compute power in hardware accelerators \cite{top500}. Hardware accelerator technologies are becoming ubiquitous in off-the-shelf products, as well. In order to take advantage of these emerging technologies, it is necessary to develop fine-grain parallelization techniques tailored for high-throughput devices such as GPUs.

For such sparse problems, pivoting becomes extremely expensive, as data management takes a large fraction of the total time compared to computation~\cite{LDLpivot}. Unfortunately, \LDLT factorization is unstable without pivoting. This is why \LDLT approaches, typically used by interior methods for nonlinear problems on CPU-based platforms \cite{tasseff2019}, have not performed as well as on hardware accelerators \cite{swirydowicz2020linear}. Some iterative methods such as MINRES~\cite{MINRES} for general symmetric matrices can make efficient (albeit memory bandwidth limited) use of GPUs because they only require matrix-vector multiplications at each iteration, which can be highly optimized, but they have limited efficiency when the number of iterations becomes large. Another approach for better-conditioned KKT systems is using a modified version of the preconditioned conjugate gradient (\textbf{PCG}) with implicit-factorization preconditioning~\cite{Dollar2006}. In our case, the ill-conditioned nature of our linear problems means that iterative methods alone are not practical~\cite{MINRES,Pyzara2011}.

We propose a hybrid direct-iterative method for solving KKT systems that is suitable for execution on hardware accelerators. The method only requires direct solves using a Cholesky factorization, as opposed to \LDLT, which means it avoids pivoting. We provide preliminary test results that show the practicality of our approach. Our test cases are generated by optimal power flow analysis \cite{chakrabarti2014security, kourounis2018toward, molzahn2017survey} applied to realistic power grid models that resemble actual grids, but do not contain any proprietary data~\cite{ACOPF}. These systems are extracted from optimal power flow analysis using Ipopt~\cite{wachter2006implementation} with MA57 as its linear solver. Solving such sequences of linear problems gives us an insight in how our linear solver behaves within an interior method. Using these test cases allowed us to assess the practicality of our hybrid approach without interfacing the linear solver with an optimization solver. 
Power grids are representative of very sparse and irregular systems commonly found in engineering disciplines.

The paper is organized as follows. \Cref{tab:notation} defines our notations. \cref{sec:nonlinear} describes the optimization problem being solved. \cref{sec:KKT} defines the linear systems that arise when an interior method is applied to the optimization problem. In \cref{sec:block2x2}, we derive a novel hybrid direct-iterative algorithm to utilize the block structure of the linear systems, and prove convergence properties for the algorithm. Numerical tests in \cref{sec:testing} show the accuracy of our algorithm on realistic systems, using a range of algorithmic parameter values.  
\Cref{sec:compare} compares our \Cpp and CUDA implementation to MA57~\cite{Duff2004}. \Cref{sec:itdirect} explains our 
decision to use a direct solver in the inner loop of our algorithm. In \cref{sec:summary}, we summarize our main contributions and results. \Cref{appA} provides supplemental figures for \cref{sec:testing}.

\label{sec:notation}
\begin{table}[t] 
\caption{\label{tab:notation} Notation. SP(S)D stands for symmetric positive (semi)definite}

\centering
\smallskip
\footnotesize

\begin{tabular}{lrrr} \toprule
   Variable  & Properties      & Functions & Meaning
\\ \midrule
   $M$ & Symmetric matrix & $\lambda_{\max}(M)$, $\lambda_{\min}(M)$ & largest, smallest (most negative) eigenvalues 
\\$M$& SPSD matrix & $\lambda_{\min *}(M)$& the smallest nonzero eigenvalue
\\$M$& SPD matrix & $\kappa(M)=\lambda_{\max}(M)/\lambda_{\min}(M)$& condition number 
\\ $J$& rectangular matrix & null($J$)& nullspace 
\\ $x$ & vector, $x>0$ & $X\equiv \text{diag}(x)$ & A diagonal matrix $X$, $X_{ii}=x_i$
\\ $e_{p}$ & vector &  &a $p$-vector of $1$s 
\\ \bottomrule
\end{tabular}
\end{table}


\section{Nonlinear optimization problem}
\label{sec:nonlinear}

We consider constrained nonlinear optimization problems of the form
\begin{subequations}\label{problemstatement}
\begin{align}
   && \min_{x\in\mathbb{R}^\nx}\ \ & f(x) \label{problemstatement_a}
\\ && \text{s.t.}        \ \ & c(x) = 0,    && \label{problemstatement_b}
\\ &&                        & d(x) \ge 0,  && \label{problemstatement_c} 
\\ &&                        &    x \ge 0,  && \label{problemstatement_d}  
\end{align}
\end{subequations}
where $x$ is an $\nx$-vector of optimization parameters, $f:~\mathbb{R}^\nx\rightarrow \mathbb{R}$ is a possibly nonconvex objective function, $c:~\mathbb{R}^\nx\rightarrow \mathbb{R}^\mc$ defines $m_c$ equality constraints, and $d:~\mathbb{R}^\nx\rightarrow \mathbb{R}^\md$ defines $m_d$ inequality constraints.
(Problems with more general inequalities can be treated in the same way.)
Functions $f(x)$, $c(x)$ and $d(x)$ are assumed to be twice continuous differentiable.
Interior methods enforce bound constraints \eqref{problemstatement_d} by adding barrier functions to the objective \eqref{problemstatement_a}:
$$ 
\min_{x\in\mathbb{R}^\nx,\,s\in\mathbb{R}^\md}
\ f(x) - \mu\sum_{j=1}^\nx \ln{x_j} - \mu\sum_{i=1}^\md \ln{s_i},
$$
where the inequality constraints \eqref{problemstatement_c}
are treated as equality constraints $d(x)-s=0$ with slack variables $s\ge 0$. The barrier parameter $\mu>0$ is  reduced toward zero using a continuation method to obtain solution that is close to the solution of \eqref{problemstatement} to within a solver tolerance.


Interior methods are most effective when exact first and second derivatives are available, as we assume for $f(x)$, $c(x)$, and $d(x)$. We define
$J(x) = \nabla c(x)$ and $J_d(x)=\nabla d(x)$ as the sparse Jacobians for the constraints. 
The solution of a barrier subproblem satisfies the nonlinear equations
\begin{subequations}\label{nonlinearequations}
\begin{align}
   \nabla f(x) + J^T y + J_d^T y_d -  z_x &= 0 
   \label{nonlinearequations_a}
\\ y_d  - z_s &=  0 
 \label{nonlinearequations_b}
\\   c(x)&=0
 \label{nonlinearequations_c}
\\    d(x)-s&=0 
\label{nonlinearequations_d}
\\ Xz_x- \mu e_\nx&=0 
\label{nonlinearequations_e}
\\ Sz_s - \mu e_\md&= 0 \label{nonlinearequations_f}, 
\end{align}
\end{subequations}
where $x$ and $s$ are primal variables, $y$ and $y_d$ are Lagrange multipliers (dual variables) for constraints \eqref{nonlinearequations_c}--\eqref{nonlinearequations_d}, and  $z_x$ and $z_s$ are Lagrange multipliers for the bounds $x\ge 0$ and $s \ge 0$.  The conditions $x>0$, $s>0$, $z_x>0$, and $z_s>0$ are maintained throughout, and the matrices $X\equiv \text{diag}(x)$ and $S\equiv \text{diag}(s)$ are SPD.

Analogously to \cite{wachter2006implementation}, at each continuation step in $\mu$ we solve nonlinear equations \eqref{nonlinearequations} using a variant of Newton's method. Typically $z_x$ and $z_s$ are eliminated from the linearized version of \eqref{nonlinearequations} by substituting the linearized versions of \eqref{nonlinearequations_e} and \eqref{nonlinearequations_f} into the linearized versions of \eqref{nonlinearequations_a} and \eqref{nonlinearequations_b}, respectively, to obtain a smaller \textit{symmetric} problem.  Newton's method then calls the linear solver to solve a series of linearized systems 
$K_k \Delta x_k = r_k$, $k=1,2,\dots$, of block $4 \times 4$ form 
\begin{align} \label{linsys4x4}
\overbrace{\begin{bmatrix}
      H + D_x     & 0         & J^T     & J_d^T
     \\ 0         & D_s       & 0           & -I    
     \\ J     & 0         & 0           & 0
     \\ J_d     & -I        & 0           & 0 
  \end{bmatrix}}^{K_k}
  \overbrace{\begin{bmatrix}
\Delta x \\ \Delta s \\ \Delta y \\ \Delta y_d
  \end{bmatrix}}^{\Delta x_k}=
    \overbrace{\begin{bmatrix}
\rtilde_{x} \\ r_s \\ r_{y} \\ r_{yd}
  \end{bmatrix}}^{r_k},
\end{align}
where index $k$ denotes optimization solver iteration (including continuation step in $\mu$ and Newton iterations), each $K_k$ is a KKT-type matrix (with saddle-point structure), vector $\Delta x_k$ is a search direction\footnote{Search directions are defined such that $x_{k+1}=x_{k}+\alpha \Delta x$ for some linesearch steplength $\alpha>0$.} for the primal and dual variables, and 
 $r_k$ is derived from the residual vector for \eqref{nonlinearequations}
 evaluated at the current value of the primal and dual variables (with $\norm{r_k} \rightarrow 0$ as the method converges):
\begin{align*}
   \rtilde_{x} &= -(\nabla f(x) + J^T y + J_d^T y_d + \mu X^{-1}e_\nx), 
\\ r_s &= - (y_d +\mu S^{-1}e_\md),
   \quad r_{y}= - c(x),
   \quad r_{yd}=s-d(x).
\end{align*}
With $Z_x\equiv  \text{diag}(z_x)$ and $Z_s\equiv \text{diag}(z_s)$, the sparse Hessian
\begin{equation*}
H \equiv \nabla^2 f(x) + \sum_{i=1}^{\mc} y_{c,i} \nabla^2 c_i(x) + \sum_{i=1}^{\md} y_{d,i} \nabla^2 d_i(x)
\end{equation*}
and diagonal $D_x \equiv X^{-1} Z_x$ are $\nx \times \nx$ matrices, $D_s \equiv S^{-1} Z_s$ is a diagonal $\md\times \md$ matrix, $J$ is a sparse $\mc\times\nx$ matrix, and $J_d$ is a sparse $\md \times \nx$ matrix.  We define $m\equiv \mc+\md$, $n\equiv \nx+\md$, and $N\equiv m+n$. 

Interior methods may take hundreds of iterations  (but typically not thousands) before they converge to a solution. All $K_k$ matrices have the same sparsity pattern, and their nonzero entries change slowly with~$k$. An interior method can exploit this
by reusing output from linear solver functions across multiple iterations $k$:
\begin{itemize}
    \item Ordering and symbolic factorization are needed only once because the sparsity pattern is the same for all $K_k$.
    \item Numerical factorizations can be reused over several adjacent Newton's iterations, e.g., when an inexact Newton solver is used within the optimization algorithm.
\end{itemize}
Operations such as triangular solves have to be executed at each iteration $k$.

The workflow of the optimization solver with calls to different linear solver functions is shown in \cref{fig:workflow}
(where $K_k \Delta x_k = r_k$ denotes the linear system to be solved at each iteration).
The main \emph{optimization} solver loop is the top feedback loop in \cref{fig:workflow}. It is repeated until the solution is optimal or a limit on optimization iterations is reached. At each iteration, the residual vector $r_k$ is updated. Advanced implementations
have control features to ensure stability and convergence of the optimization solver. The lower feedback loop in \cref{fig:workflow} shows linear system regularization by adding a diagonal perturbation to the KKT matrix. One such perturbation removes singularity  \cite[Sec.~3.1]{wachter2006implementation}, which happens if there are redundant constraints. The linear solver could take advantage of algorithm control like this and request matrix perturbation when beneficial.

\begin{figure*}[t]   
\centering
  \includegraphics[width=0.9\textwidth]{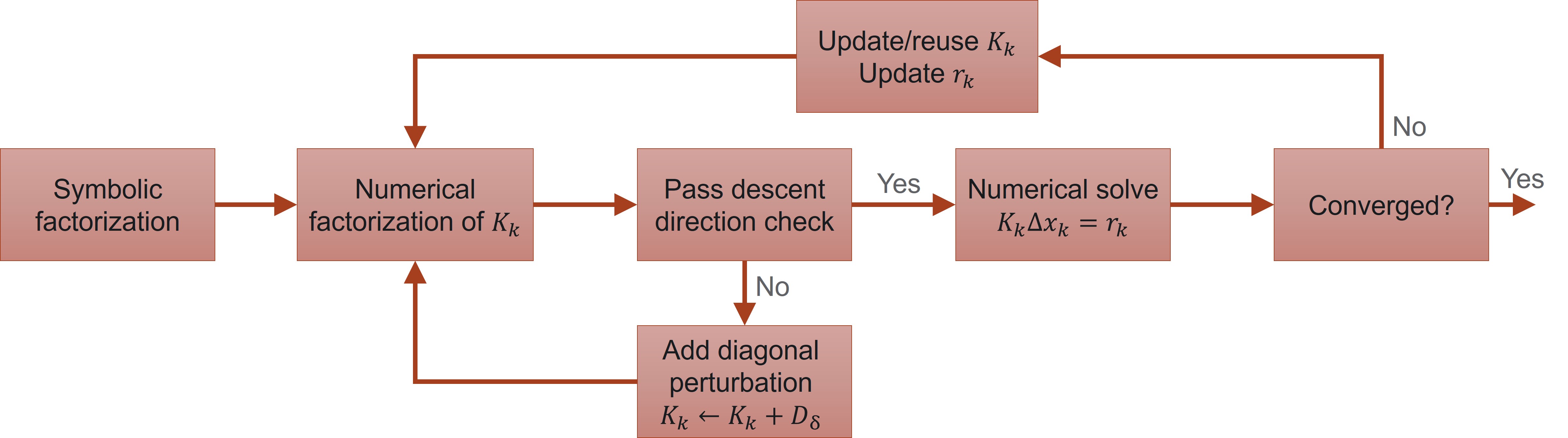}
  \caption{Optimization solver workflow showing invocation of key linear solver functions. The top feedback loop represents the main optimization solver iteration loop. The bottom feedback loop is the optimization solver control mechanism to regularize the underlying problem when necessary.}
  \label{fig:workflow}
\end{figure*}

\section{Solving KKT linear systems}
\label{sec:KKT}
 $ $\LDLT factorization via MA57 \cite{Duff2004} has been used effectively for extremely sparse  problems on traditional CPU-based platforms, but is not suitable for fine grain parallelization required for GPU acceleration. 
 Parallel and GPU accelerated direct solve implementations such as SuperLU~\cite{Li2005,SuperLUweb}, STRUMPACK~\cite{Rouet2016, SSTRUMPACKweb}, and PaStiX~\cite{PaStiX,pastixweb} exist for general symmetric indefinite systems (although the first two are designed for general systems), but these software packages are designed to take advantage of dense blocks of the matrices in denser problems and do not perform well on our systems of interest, which do not yield these dense blocks~\cite{swirydowicz2020linear,he2015gpu}.

The fundamental issue with using GPUs for $\LDLT$ is that this factorization is not stable without pivoting~\cite{GV4}. Pivoting requires considerable data movement and, as a result, a substantial part of the run time is devoted to memory access and communication. Any gains from the hardware acceleration of floating point operations are usually outweighed by the overhead associated with permuting the system matrix during the pivoting.  This is especially burdensome because both rows and columns need to be permuted in order to preserve symmetry~\cite{LDLpivot}. While any of the two permutations can be performed efficiently on its own with an appropriate data structure for the system's sparse matrix (\textit{i.e.}, row compressed sparse storage for row permutations and, analogously, column compressed sparse storage for column permutations) swapping both rows and columns  simultaneously  is necessarily costly.

Here we propose a method that  uses sparse Cholesky factorization (in particular, a GPU implementation). Cholesky factorization is advantageous for a GPU implementation because it is stable without pivoting and can use GPUs efficiently compared to \LDLT~\cite{RENNICH2016140}. Furthermore, the ordering of the unknowns (also known as the symbolic factorization) for the purpose of reducing fill-in in the factors can be established without considering the numerical values and only once at the beginning of the optimization process, and hence, its considerable computational cost is amortized over the optimization iterations.

To make the problem smaller, we can eliminate $\Delta s = J_d \Delta x - r_{yd}$
and $\Delta y_d = D_s \Delta s - r_s $ from \eqref{linsys4x4} to obtain the $2 \times 2$ system~\cite[Sec.~3.1]{petra2009computational} 
\begin{align} \label{linsys2x2}
  \bmat{\Htilde & J^T
    \\  J     & 0}
  \bmat{\Delta x \\ \Delta y}
  = \bmat{r_x \\ r_{y}},
  \quad
  \Htilde \equiv H + D_x + J_d^T D_s J_d, 
\end{align}
where $r_x = \rtilde_x + J_d^T (D_s r_{yd} + r_s )$. 
This reduction requires block-wise Gaussian elimination with
block pivot {\footnotesize$\bmat{D_s & -I \\ -I}$},
which is ill-conditioned when $D_s$ has large elements,
as it ultimately does. Thus, system \eqref{linsys2x2} is smaller but more ill-conditioned.
After solving \eqref{linsys2x2}, we compute $\Delta s$ and $\Delta y_d$ in turn to obtain the solution of \eqref{linsys4x4}. 

\section{A block $2\times 2$ system solution method}
\label{sec:block2x2}
 Let $Q$ be any SPD matrix. Multiplying the second row of \eqref{linsys2x2} by $J^TQ$ and adding it to the first row gives a system of the form
\begin{align} \label{linsys2x2mod}
  \bmat{\Hgam & J^T \\ J & 0}
  \bmat{\Delta x \\ \Delta y}
 =
  \bmat{\rhat_x \\ r_{y}},
  \qquad \Hgam= \Htilde + J^TQJ,
\end{align}
where $\rhat_{x}=r_x+J^TQr_{y}$. The simplest choice is $Q=\gamma I$ with $\gamma>0$:
\begin{equation}\label{eq:gamma}
    \Hgam = \Htilde + \gamma \, J^T J.
\end{equation}
When $\Hgam$ is SPD, its Schur complement $S\equiv J\Hgam\inv J^T$ is well defined, and \eqref{linsys2x2mod} is equivalent to
\begin{align}
   S\Delta y   &= J\Hgam\inv\rhat_x-r_{y}, \label{Schur1} \\ 
   \Hgam\Delta x &= \rhat_x-J^T\Delta y.    \label{Schur2}
\end{align}
This is the approach of Golub and Greif \cite{GG03} for saddle-point systems (which have the structure of the $2\times 2$ system in \eqref{linsys2x2}). Golub and Greif found experimentally that $\gamma=\norm{\Htilde}/\norm{J}^2$ made $\Hgam$ SPD and better conditioned than smaller or larger values.
%
We show in Theorems \ref{thm:2} and \ref{thm:4}
that for large $\gamma$, the condition number $\kappa(\Hgam)$ increases as $\gamma\rightarrow\infty$, but $\kappa(S)$ converges to $1$ as $\gamma\rightarrow\infty$. \Cref{cor:1} shows there is a finite value of $\gamma$ that minimizes $\kappa(\Hgam)$.  This value is probably close to $\gamma=\norm{\Htilde}/\norm{J}^2$.

Our contribution is to combine the system reduction in~\cite{petra2009computational} (from \eqref{linsys4x4} to \eqref{linsys2x2}) with the method of~\cite{GG03} for changing \eqref{linsys2x2} to \eqref{linsys2x2mod} to solve an optimization problem consisting of a series of block $4\times 4$ systems using a GPU implementation of sparse Cholesky factorization applied to $\Hgam$. A new method for regularizing is added, and practical choices for parameters are given based on scaled systems. Also, important convergence properties of the method are proven in Theorems \ref{thm:1}--\ref{thm:4}.

If \eqref{linsys2x2mod} or $\Hgam$ are poorly conditioned, the only viable option may be to ignore the block structure in \eqref{linsys2x2mod} and solve \eqref{linsys4x4} with an \LDLT factorization such as MA57 (likely without help from GPUs). This is the fail-safe option. Otherwise, we require $\Hgam$ to be SPD (for large enough $\gamma$) and use its Cholesky factorization to apply the conjugate gradient method (CG) \cite{CG} or MINRES \cite{MINRES} to \eqref{Schur1} with or without a preconditioner.
If the $R$ part of QR factors of $J^T$ is not too dense,
we could use
$$
M \equiv (JJ^T)\inv J\Hgam J^T(JJ^T)\inv
$$
as a multiplicative preconditioner. This gives the exact solution if the RHS is orthogonal to null($J$) or $J$ is square, which is not possible in our case. In our experiments, solutions with the preconditioner $M$ lagged slightly behind those without it, but both take $O(1)$ iterations. We proceed without the use of preconditioner.

\subsection{A hybrid solver with minimal regularization}
\label{sec:algorithm}
Typically, \eqref{linsys2x2} starts with an SPD $\Htilde$ and full-rank $J$. As the optimization solver iterations progress, $\Htilde$ may become indefinite and $J$'s rank may shrink (at least numerically).  This means the system becomes singular and must be regularized. We seek a small regularization to avoid changing too much the solution of the equivalent system \eqref{linsys2x2mod}.

An SPD $\Hgam$ guarantees that $\Htilde$ is SPD on null($J$), a requirement at the solution of the optimization problem.

\begin{theorem} 
\label{thm:1}
 For large enough $\gamma$ ($\gamma>\gamma_{\min}$) and full row rank $J$, $\Hgam=\Htilde + \gamma J^TJ$ is SPD iff $\Htilde$ is positive definite on null($J$).
\end{theorem}
\begin{proof}
Assume $\Hgam$ is SPD. For any nonzero vector $v_0$ in null($J$), we have
  $v_0^T\Htilde v_0 = v_0^T\Hgam v_0 > 0$ (this direction of the proof does not require $J$ to have full row rank).

Conversely, assume $\Htilde$ is positive definite on null($J$) and $J$ has full row rank. For any nonzero vector $v=v_0+v_1$ with $v_0$ in null($J$) and $v_1$ orthogonal to null($J$),
\begin{align*}
  v^T(\Htilde + \gamma J^TJ)v=
  v_0^T\Htilde v_0 + v_1^T(\Htilde + \gamma J^TJ)v_1.
\end{align*}
We have $v_0^T\Htilde v_0\ge 0$ by assumption. Further,
\begin{align*}
  v_1^T(\Htilde + \gamma J^TJ)v_1 \ge \big(\lambda_{\min}(\Htilde)
 + \gamma \lambda_{\min*}(J^TJ) \big) v_1^Tv_1>0 
\end{align*}
if $\gamma \geq\gamma_{\min}\equiv -\lambda_{\min}(\Htilde)/\lambda_{\min*}(J^TJ)$.
\end{proof}

We work with $\gamma \geq 0$. We use $\Hgam$ SPD 
as a proxy for $\Htilde$ being SPSD on null($J$), keeping in mind that if it does not hold even for a very large $\gamma$ in \eqref{eq:gamma}, $\Hgam$ is singular and needs to be modified. However, $\gamma$ cannot be made arbitrarily large without increasing $\kappa(\Hgam)$ when $J$ is rectangular, as in our case. There must be an ideal intermediate value of $\gamma$.

\begin{theorem} 
\label{thm:2}
If $J$ has full row rank with more columns than rows, $\Htilde$ is symmetric and positive definite on null($J$), and $\Hgam=\Htilde+\gamma J^TJ$, there exists $\gamma_{\max}\geq \max(\gamma_{\min},0)$ such that for $\gamma \geq \gamma_{\max}$, $\kappa(\Hgam)$ increases linearly with $\gamma$.
\end{theorem}
\begin{proof} 
\begin{gather*}
    \lambda_{\max}(\Hgam)\equiv \max_{\norm{v}_2=1} v^T\Hgam v\leq \max_{\norm{v}_2=1} v^T\Htilde v + \max_{\norm{v}_2=1} v^T(\gamma J^TJ) v \\
    = \lambda_{\max}(\Htilde)+ \gamma\lambda_{\max}(J^TJ),
    \\
    \lambda_{\max}(\Hgam)\geq \min_{\norm{v}_2=1} v^T\Htilde v + \max_{\norm{v}_2=1} v^T(\gamma J^TJ) v=\lambda_{\min}(\Htilde)+ \gamma\lambda_{\max}(J^TJ).
\end{gather*}
Hence $\lambda_{\min}(\Htilde)+ \gamma \lambda_{\max}(J^TJ)\leq \lambda_{\max}(\Hgam)\leq \lambda_{\max}(\Htilde)+ \gamma \lambda_{\max}(J^TJ)$, meaning $\lambda_{\max}(\Hgam) \propto \gamma$ for large enough $\gamma$ (defined as $\gamma \geq \gamma_{\max}\geq \max(\gamma_{\min},0)$).
Similarly, 
\begin{gather*}
    \lambda_{\min}(\Hgam)\equiv \min_{\norm{v}_2=1} v^T\Hgam v\geq \min_{\norm{v}_2=1} v^T\Htilde v + \min_{\norm{v}_2=1} v^T(\gamma J^TJ)v =\lambda_{\min}(\Htilde),
    \\
    \lambda_{\min}(\Hgam)\leq \max_{\norm{v}_2=1} v^T\Htilde v + \min_{\norm{v}_2=1} v^T(\gamma J^TJ) v=\lambda_{\max}(\Htilde).
\end{gather*}
Thus $\lambda_{\min}(\Htilde)\leq \lambda_{\min}(\Hgam)\leq \lambda_{\max}(\Htilde)$. From \cref{thm:1}, $\gamma\geq \gamma_{\min}\Rightarrow\lambda_{\min}(\Hgam)>0$,  so that $\kappa(\Hgam)=\lambda_{\max}(\Hgam)/\lambda_{\min}(\Hgam)\propto \gamma$ for $\gamma\geq \gamma_{\max}$.
\end{proof}

\begin{corollary}
\label{cor:1}
Among $\gamma$s such that $\Hgam$ is SPD ($\gamma\ge \gamma_{\min}$), $\kappa(\Hgam)$ is minimized for some $\gamma \in [\gamma_{\min},\gamma_{\max}]$. 
\end{corollary}

In practice, the optimizer may provide  systems where $\Htilde$ is not SPD on null($J$). In this case we can regularize $\Hgam$ by using $\Hdel=\Hgam + \delta_1 I$ instead.
Unlike $\gamma$, the introduction of $\delta_1$ changes the solution of the system, so it is essential to keep $\delta_1$ as small as possible. If $\Hgam$ is not SPD, we set $\delta_1=\delta_{\min}$, a parameter for some minimum value of regularization. If $\Hdel$ is still not SPD, we double $\delta_1$ until it is. This ensures we use the minimal value of regularization (to within a factor of 2) needed to make $\Hdel$ SPD. 

If $\delta_1$ proves to be large, which can happen before we are close to a solution, it is essential for the optimizer to be informed to allow it to modify the next linear system.
When the optimizer nears a solution, $\delta_1$ will not need to be large.

In our tests, $\delta_1$ starts at $0$
for the next matrix in the sequence, but is set back to its previous value if the factorization fails.  

We also set $\delta_{\max}$, the maximum allowed $\delta_1$ before we resort to \LDLT factorization of $K_k$ or return to the optimization solver.

If $J$ has low row rank, $S$ in \eqref{Schur1} is SPSD. In this case, CG will succeed if \eqref{Schur1} is consistent. Otherwise, it will encounter a near-zero quadratic form. We then restart CG on the regularized system $(J \Hdel\inv J^T+\delta_2 I)\Delta
y=Jw-r_{y}$. In this way, $\delta_1$ regularizes the $(1,1)$ block and $\delta_2$ regularizes the $(2,2)$ block.

 To ensure we can judge the size of parameters $\gamma$ and $\delta_{\min}$ relative to system \eqref{linsys2x2mod}, we first scale \eqref{linsys2x2mod} with a symmetrized version of Ruiz scaling \cite{Ruiz01ascaling}.

\Cref{alg:CGschur} is a generalization of the Uzawa iteration~\cite{Uzawa} for KKT systems with a (1,1) block that is not necessarily positive definite. It gives a method for solving a sequence of systems \{\eqref{Schur1}--\eqref{Schur2}\} with $Q=\gamma I$ used in the calculation of $\Hgam$. The workflow is similar to~\cref{fig:workflow} except only $\Hdel$ is factorized. On lines 16--19, $\Hdel\inv$ is applied by direct triangular solves with the Cholesky factors $L$, $L^T$ of $\Hdel$. Each iteration of the CG solve on line 17 requires multiplication by $J^T$, applying $\Hdel\inv$, multiplying by $J$, and adding a scalar multiple of the original vector.
\Cref{sec:itdirect} shows why complete Cholesky factorization of $\Hdel$ was chosen.

The rest of the section discusses other important bounds that $\gamma$ must obey. However, selecting an ideal $\gamma$ in practice is difficult and requires problem heuristics (like $\gamma=\norm{\Htilde}/\norm{J}^2$ in \cite{GV4}) or trial and error.

\begin{algorithm}[t]
\begin{algorithmic}[1]

\FOR {each matrix in the sequence such as in \eqref{linsys2x2mod}}
\STATE  $\delta_1\gets 0$
\STATE $\Hdel= \Hgam$ ($\gamma$ used in the calculation of $\Hgam$)
\STATE Try $LL^T=\texttt{chol}(\Hdel)$ (Fail $\gets$ False if factorized, True otherwise)
\WHILE {Fail and $\delta_1<=\delta_{\max}/2$} 
\IF {$\delta_1==0$}
 \STATE$\delta_1\gets \delta_{\min}$
\ELSE 
\STATE $\delta_1=\delta_{\min}\gets 2\delta_{\min}$

\ENDIF
\STATE $\Hdel= \Hgam+\delta_1 I$
\STATE Try $LL^T=\texttt{chol}(\Hdel)$ (Fail $\gets$ False if factorized, True otherwise)
\ENDWHILE
\IF{Fail==False}
\STATE Direct solve $\Hdel w=\rhat_x$ 
\IF {CG on $(J \Hdel\inv J^T)\Delta
y=Jw-r_{y}$ produces a small quadratic form}
\STATE CG solve $(J \Hdel\inv J^T+\delta_2 I)\Delta
y=Jw-r_{y}$ (perturbed \eqref{Schur1})
\ENDIF
\STATE Direct solve \hbox{$\Hdel \Delta x = \rhat_x-J^T\Delta y$\hspace*{65pt}} (perturbed \eqref{Schur2})
\ELSE
\STATE Use \LDLT to solve \eqref{linsys4x4} or return problem to optimization solver
\ENDIF
\ENDFOR
\end{algorithmic}
\caption{Using CG on Schur complement to solve the block system \eqref{linsys2x2mod} by solving \eqref{Schur1}--\eqref{Schur2}. $\Hdel$ is a nonsingular perturbation of $\Hgam=\Htilde + \gamma J^TJ$.
Typical parameters: $\gamma = 10^{4}$, $\delta_{\min}=\delta_{2}= 10^{-9}$, $\delta_{\max}=10^{-6}$.}
\label{alg:CGschur}
\end{algorithm}

\subsection{Guaranteed descent direction}

Optimization solvers typically use an \LDLT factorization to solve the $N \times N$ system \eqref{linsys4x4} at each step because (with minimal extra computation) it supplies the inertia of the matrix. A typical optimization approach treats each of the four $2\times 2$ block of~\eqref{linsys4x4} as one block, accounting for the possible regularization applied to $\Hgam$.
We mean that the (1,1) block {$H_K\equiv$\footnotesize$\bmat{ H+D_x+\delta_1 I&  \\ & D_s}$} is a Hessian of dimension $n$, and the (2,1) block {$J_K\equiv$\footnotesize$\bmat{ J&  \\ J_d& -I}$} represents the constraint Jacobian and has dimensions $m \times n$. The inertia being $(n,m,0)$ implies that (a) $H_K$ is uniformly SPD on null($J_K$) for all $k$, meaning $v^TH_Kv\geq \epsilon >0$ for all vectors $v$ satisfying $J_Kv=0$, iterations $k$, and some positive constant~$\epsilon$; (b) KKT matrix in \eqref{linsys4x4} is nonsingular. Together these conditions are sufficient to ensure a descent direction and fast convergence in the optimization algorithm~\cite{wachter2006implementation,NocW2006}. We show that \cref{alg:CGschur} ensures properties (a) and (b) without computing the inertia of the KKT matrix. %

\begin{theorem} 
\label{thm:3}
If \cref{alg:CGschur} succeeds with $\delta_2=0$, it provides a descent direction for the interior method for large enough $\gamma$.
\end{theorem}

\begin{proof}
By construction, $\Hdel$ is uniformly SPD (because the Cholesky factorization was successful for all iterations of the solver until this point) and the KKT system in \eqref{linsys2x2mod} is nonsingular (because $J$ has full row rank, or $\delta_2 I$ was added to $S$ to shift it from SPSD to SPD). 
Therefore, \cref{alg:CGschur} provides a descent direction for \eqref{linsys2x2mod} even if regularization is added. The rest of the proof assumes $\delta_2=0$, and we return to the other case later. Since $J$ has full row rank, $H+D_x+ \delta_1 I$ is nonsingular by assumption, all block rows in~\eqref{linsys4x4} have full rank internally, and Gaussian elimination cannot cause cancellation of an entire block row, we conclude that \eqref{linsys4x4} is nonsingular. Let $\Hkgam\equiv H_K+\gamma_K J_K^TJ_K$ (for $\gamma_K\geq \gamma$ used in \cref{alg:CGschur}). For any nonzero vector $u^T=(u_1^T,  u_2^T)$ with $u_1$ and $u_2$ of dimensions $\nx$ and $\md$, 
$$
    u^T \Hkgam u = u_1^T (H+D_x+\gamma_K J^TJ)u_1 + u_2^TD_su_2+ \gamma_K w^Tw,
$$
where $w=J_du_1-u_2$. If $w=0$, then $u_2=J_du_1$, which means $u^T\Hkgam u\geq u_1^T\Hgam u_1\geq \epsilon>0$ and the proof is complete (with $\gamma_K=\gamma$). Otherwise, $\gamma_w\equiv \gamma_K w^Tw>0$. So for large enough $\gamma_K$,  $u^T\Hkgam u\geq \epsilon>0$. Applying \cref{thm:1} to \eqref{linsys4x4} with $H_K$ and $J_K$ replacing $\Htilde$ and $J$ shows that $H_K$ is positive definite on null($J_K$).
\end{proof}

We note that $\delta_1$ corresponds to the so-called primal regularization of the filter line-search algorithm~\cite{wachter2006implementation}. Under this algorithm, whenever $\delta_1$ becomes too large, one can invoke the feasibility restoration phase of the filter line-search algorithm~\cite{wachter2006implementation} as an alternative to performing the  $\LDLT$ factorization on the CPU. Feasibility restoration ``restarts'' the optimization at a point with more favorable numerical properties. We also note that when $\delta_1$ is sufficiently large, the curvature test used in \cite{Chiang2016} should be satisfied. Hence, inertia-free interior methods have the global convergence property without introduction of other regularization from the outer loop. 

The $\delta_2$ regularization is a numerical remedy for low-rank $J$, caused by redundant equality constraints. This regularization is similar to the so-called dual regularization used in Ipopt~\cite{wachter2006implementation} and specifically addresses the issue of rank deficiency; however, there is no direct analogue from a $\delta_2$ regularization in \eqref{linsys2x2mod} to Ipopt's dual regularization in~\eqref{linsys4x4} and neither of the two heuristics guarantees a descent direction for~\eqref{linsys4x4}. Given the similarity between the two heuristics, we believe that the $\delta_2$ regularization can be effective within the filter line-search algorithm.

When \cref{alg:CGschur} is integrated with a nonlinear optimization solver such as \cite{waecther_05_ipopt} and \cite{HiOp}, the while loop (Line 5) in \cref{alg:CGschur} can be removed and the computation of $\delta_1$ and $\delta_2$ can be decided by the optimization algorithm. The development of robust heuristics that allow \cref{alg:CGschur}'s $\delta_1$ and $\delta_2$ regularization  within the filter line-search algorithm will be subject of future work.

\subsection{Convergence for large $\gamma$}
In \cref{sec:algorithm} we showed that $\Hgam$ is SPD for large enough $\gamma$, and that $\gamma$ should not be so large that the low-rank term $\gamma (J^TJ)$ makes $\Hgam$ ill-conditioned.  Here we show that in order to decrease the number of CG iterations, it is beneficial to increase $\gamma$ beyond what is needed for $\Hgam$ to be SPD.

\begin{theorem} 
\label{thm:4}
 In exact arithmetic, for $\gamma \gg 1$, the eigenvalues of $S_\gamma \equiv \gamma S$ converge to $1$ with an error term that decays as $1/\gamma$.  
\end{theorem}

\begin{proof}
By definition,
\begin{align}
   S_\gamma = \gamma J(\Htilde +\gamma J^TJ)\inv J^T=J\left (\frac{\Htilde}{\gamma}+J^TJ\right)\inv J^T.
   \label{eq:S1}
\end{align}
  Since $\Htilde$ is nonsingular and $J$ has full row rank by assumption,
  the Searle identity $(A+BB^T)\inv B^T=A\inv  B(I+B^TA\inv B)\inv$~\cite{Searle}
with $A\equiv \Htilde/\gamma$ and $B\equiv J^T$
gives
\begin{align*}
   S_\gamma=\gamma J\Htilde\inv J^T(I+\gamma J\Htilde\inv J^T)\inv=\left(I+C\right)\inv,
   \quad
   C \equiv \frac{1}{\gamma} \left(J\Htilde\inv J^T\right)\inv.
\end{align*}
For $\gamma \rightarrow \infty$, $\left|\lambda_{i} (C)\right| \ll 1 \ \forall i$. 
The identity $(I+C)\inv =\Sigma_{k=0}^{\infty} (-1)^k C^k$, which applies when $\left|\lambda_{i} (C)\right| < 1 \ \forall i$, gives
\begin{align}
   S_\gamma=\sum_{k=0}^{\infty}(-1)^k C^k=I - C + O\left(\frac{1}{\gamma^2}\right).
   \label{eq:S3}
\end{align}
\end{proof}

A different proof of an equivalent result is given by Benzi and Liu \cite[Lemma 3.1]{Benzi07}. 
\begin{corollary}
The eigenvalues of $S$ are well clustered for $\gamma \gg 1$, and the iterative solve in \cref{alg:CGschur} line 16 converges quickly. 
\end{corollary}

In the next section, we show that our choices of $\gamma=10^4$--$10^6$ are large enough for the arguments to hold.  This explains the rapid convergence of CG. We also show that our transformation from \eqref{linsys4x4} to \eqref{linsys2x2mod} and our scaling choice are stable on our test cases, by measuring the error for the original system \eqref{linsys4x4}. 

\section{Practicality demonstration}
\label{sec:testing}

We demonstrate the practicality of \cref{alg:CGschur} using five series of linear problems \cite{ACOPF} generated by the Ipopt solver performing optimal power flow analysis 
on the power grid models
summarized in \cref{tab:descr}. We compare our results with a direct solve via MA57's \LDLT factorization~\cite{Duff2004}.    

\begin{table}[b] 
\caption{\label{tab:descr} Characteristics of the five tested optimization problems, each generating sequences of linear systems $K_k \Delta x_k = r_k$  \eqref{linsys4x4} of dimension $N$. Numbers are rounded to 3 digits. K and M signify $10^3$ and $10^6$.}

\centering
\footnotesize

\begin{tabular}{lrr} \toprule
   Name        & $N(K_k)$~~~ & \nnz($K_k$)
\\ \midrule
   South Carolina grid     &    56K &   411K
\\ Illinois grid  &  4.64K &  21.6K
\\ Texas grid &  55.7K &   268K
\\ US Western Interconnection grid  &   238K &  1.11M
\\ US Eastern Interconnection grid  &  1.64M &  7.67M
\\ \bottomrule
\end{tabular}
\end{table}

\subsection{$\gamma$ selection}

\begin{figure}[t]   
\centering
  \includegraphics[width=\Width\textwidth]{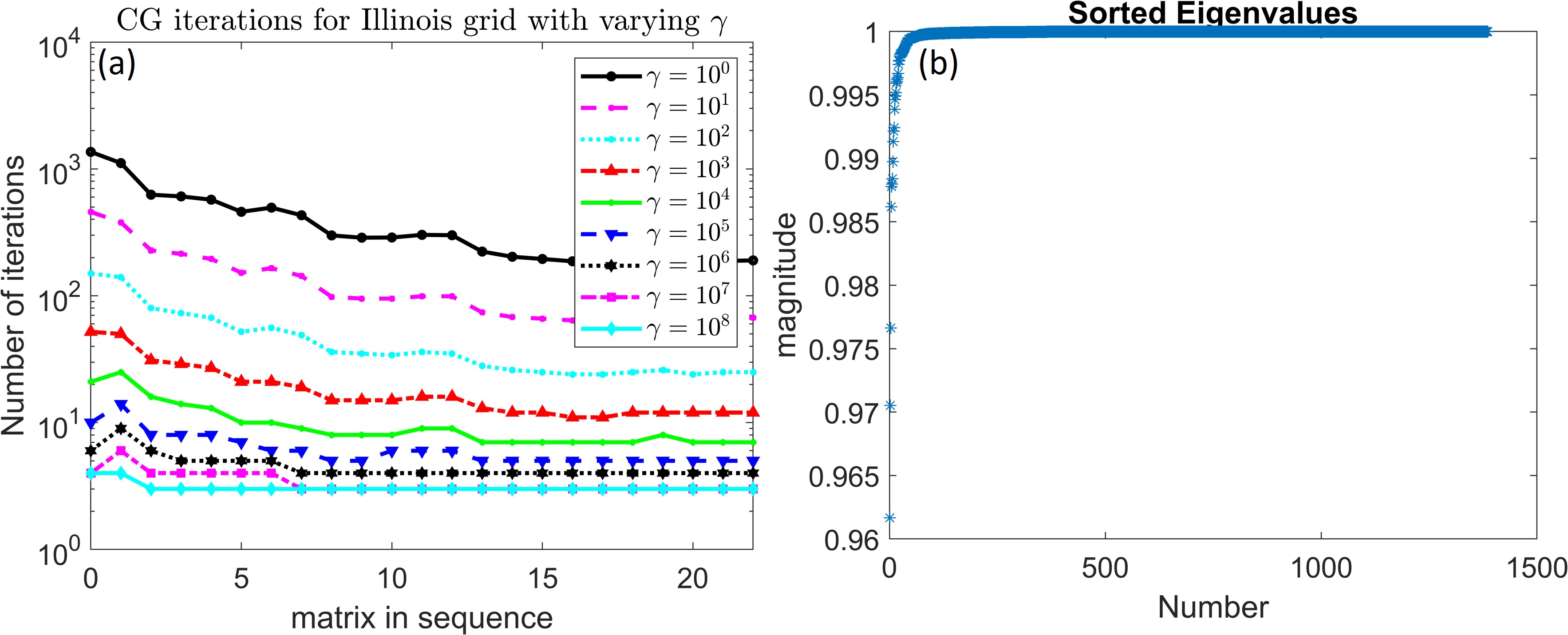}
  \caption{Illinois grid: (a) CG iterations on \cref{Schur1} with varying $\gamma$. $\gamma \geq 10^3$ gives good convergence. The mean number of iterations for $\gamma=10^4$ is $9.4$. (b) Sorted eigenvalues of $S_\gamma \equiv \gamma J\Hgam\inv J^T$ in \eqref{eq:S1} matrix 22, for $\gamma=10^4$. The eigenvalues are clustered close to $1$.}
  \label{fig:alphas200}
\end{figure}

\begin{figure}[tbp]   
\centering
  \includegraphics[width=\Width\textwidth]{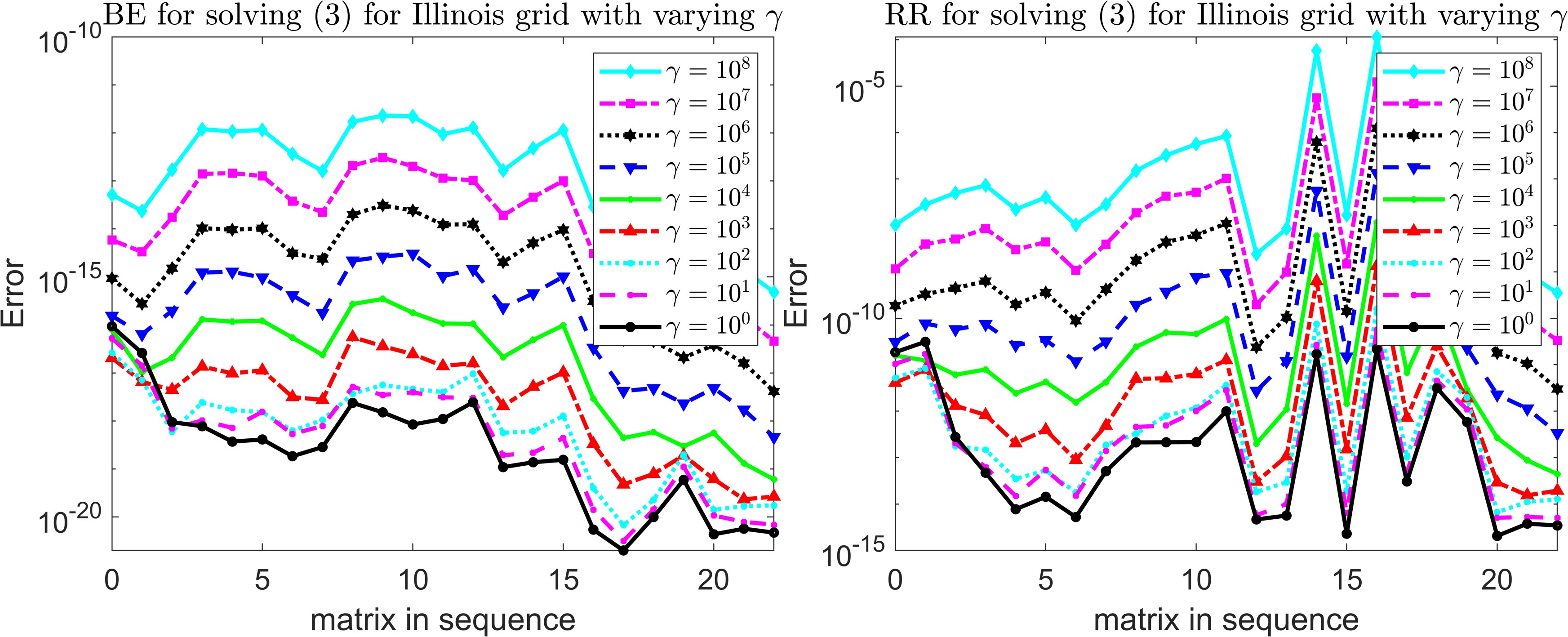}
  \caption{Illinois grid \eqref{linsys4x4}, with varying $\gamma$ in \eqref{eq:gamma}: (a) Backward error (\textbf{BE}) and (b) relative residual (\textbf{RR}). (a) $\gamma \leq 10^4$ gives results close to machine precision. (b) $\gamma \leq 10^4$ has $\text{RR} \le 10^{-8}$. }
  \label{fig:alphas200back44low}
\end{figure}

We use \cref{alg:CGschur} with CG as the iterative method on line 16. A larger $\gamma$ in $\Hgam=\Htilde + \gamma J^TJ$ may improve CG convergence but make $\Hgam$ more ill-conditioned and increase the error in the solution of \eqref{linsys4x4}.  We therefore run some preliminary tests to find a suitable $\gamma$ before selecting $\delta_{\min}$ and testing other problems,
and we require CG to solve accurately (stopping tolerance $10^{-12}$ for the relative residual norm). We start with one of the smaller problems, the Illinois grid, to eliminate some values of $\gamma$ and see if the condition number of $K_k$ for each matrix in the sequence is useful in determining the needed iterations or regularization. We run these tests with $\delta_{\min}=10^{-10}$.  

\Cref{fig:alphas200}(a) shows that for the Illinois grid sequence, values of $\gamma \ge 10^4$ give CG convergence in approximately 10 iterations for every matrix in the sequence.  For the last matrix we found that eigenvalues of $S$ in \eqref{Schur1} are very well clustered and the condition number of $S$ is $\kappa \approx 1.04$, as shown in \cref{fig:alphas200}(b).  This guarantees and explains the rapid convergence because for CG applied to a general SPD system $Ax=b$,
\begin{align}
  \frac{\norm{e_k}_A}{\norm{e_0}_A} \le 2\left(\frac{\sqrt{\kappa}-1}{\sqrt{\kappa}+1}\right)^k,    \label{CGconv}
\end{align}
where $e_k=x-x_k$ is the error in an approximate solution $x_k$ at iteration~$k$ \cite{GV4}.

For the last few matrices, $\gamma=10^8$ is the only value requiring $\delta_1>0$.  The final value of $\delta_1$ was $16\delta_{\min} = 1.6 \cdot 10^{-9}$.
No important information was gleaned from $\texttt{cond}(K_k)$. For all other values of $\gamma$, $\delta_1=0$ for the whole sequence. For a system $Ax=b$ and an approximate solution $\xtilde \approx x$, we define the backward error \textbf{BE} as $\norm{A\xtilde-b}_2/(\norm{A}_2 \, \norm{\xtilde}_2+\norm{b}_2)$ and the relative residual \textbf{RR} as $\norm{A\xtilde-b}_2/\norm{b}_2$. As is  common practice, we use $\norm{A}_\infty$ to estimate $\norm{A}_2$, which is too expensive to calculate directly. $\norm{A}_\infty$  always provides an upper bound for $\norm{A}_2$, but in practice is quite close to the actual value. 
Note that MA57 always has a BE of order machine precision.
\Cref{fig:alphas200back44low} 
shows the (a) BE and (b) RR for system \eqref{linsys4x4} for varying $\gamma$. Results for the BE and RR of system \eqref{linsys2x2} are not qualitatively different and are given in \cref{appA}. One conclusion is that increasing $\gamma$ to reduce CG iterations can be costly for the accuracy of the solution of the full system. Based on the results of this section, $\gamma$ in the range $10^2$--$10^6$ gives reasonable CG iterations and final accuracy. For other matrices, we present a selected $\gamma$ in this range that produced the best results.

\subsection{Results for larger matrices}

\begin{figure}[t]   
\centering
  \includegraphics[width=0.45\textwidth]{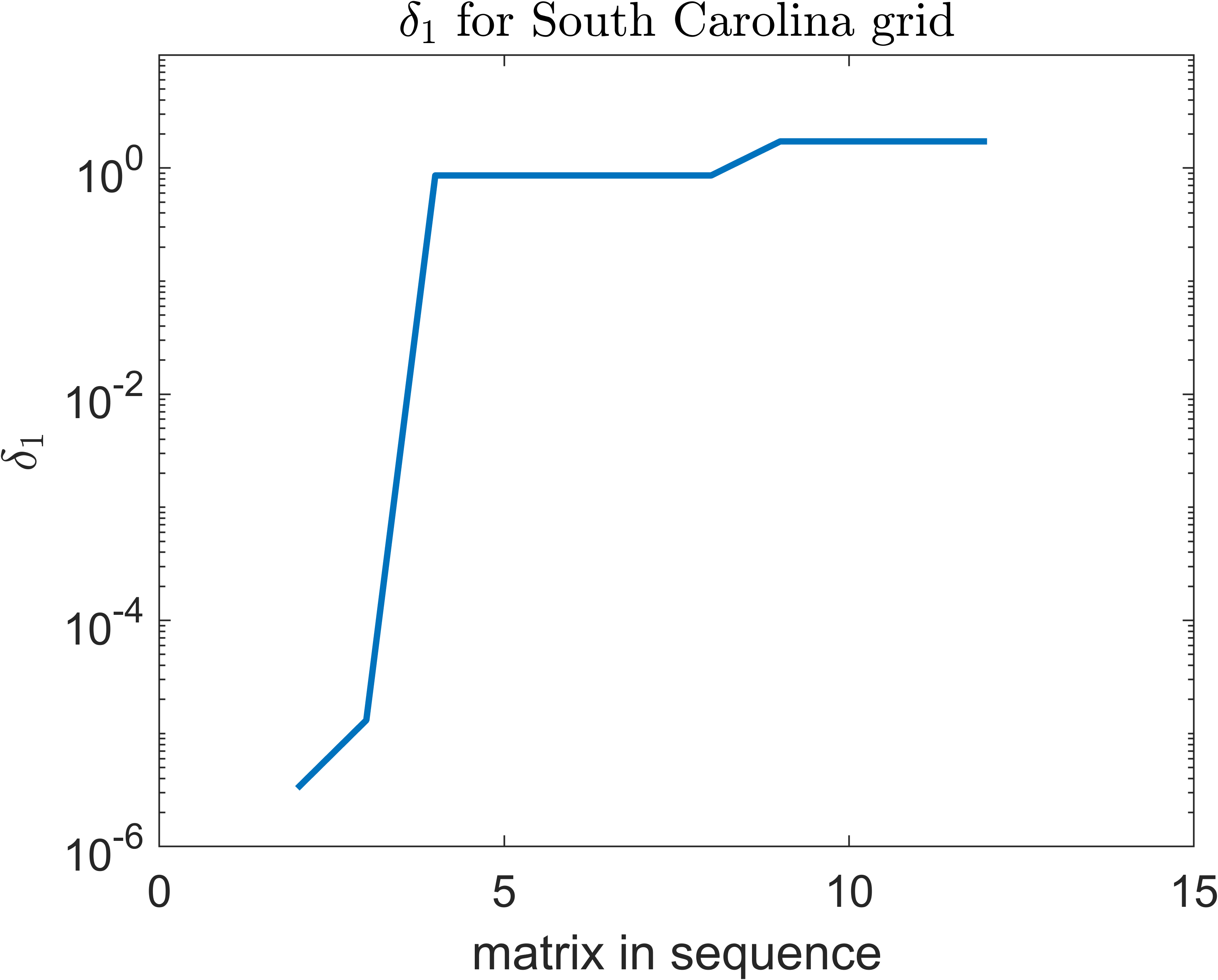}
  \caption{South Carolina grid: $\delta_1$ for $\gamma=10^4$. For other values of $\gamma$ the graph was similar. Except for the first few and last few matrices, $\delta\approx 1$, meaning the required regularization would make the solution too inaccurate. The value of $0$ is omitted on the log-scale. }
  \label{fig:South Carolina grid}
\end{figure}

\begin{figure}[t]   
\centering
  \includegraphics[width=\Width\textwidth]{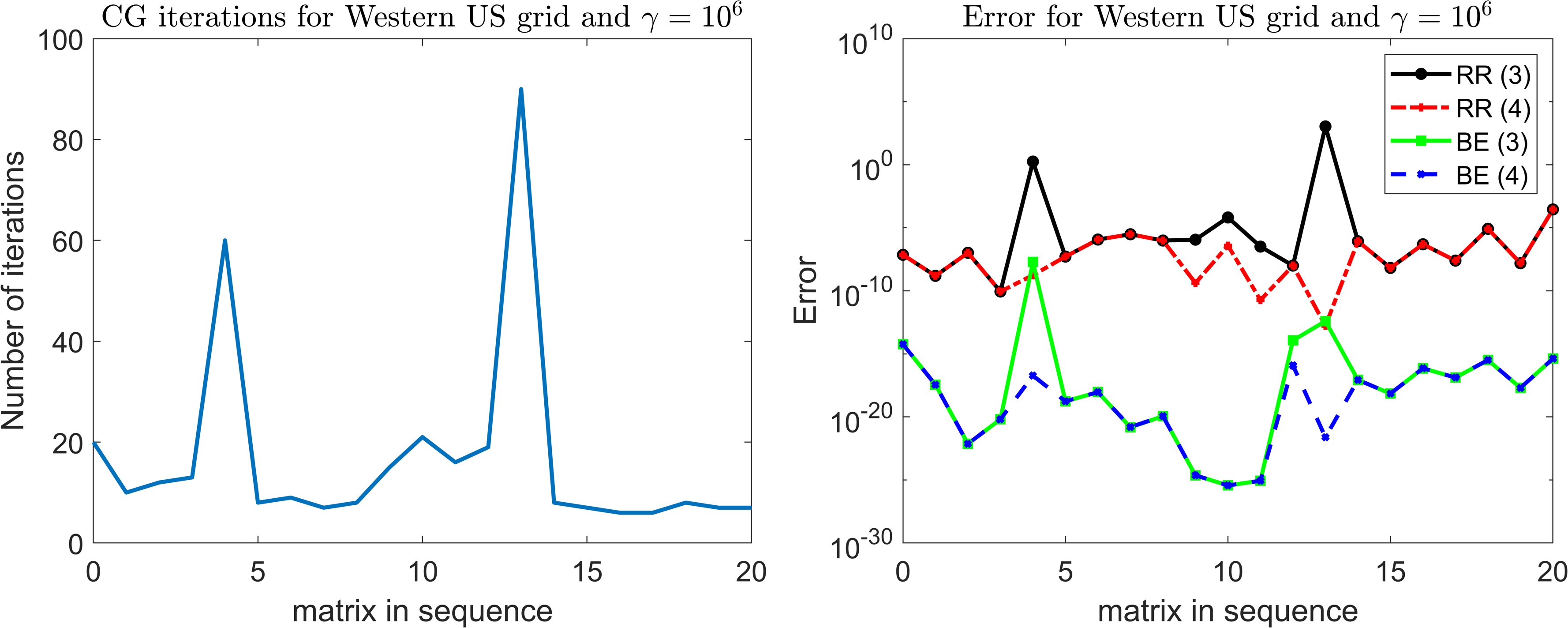}
  \caption{US Western Interconnection grid with $\gamma=10^6$ in \eqref{eq:gamma}: (a) CG iterations on \cref{Schur1}. The mean number of iterations is $17$. (b) BE and RR for the sequence. The BE for \eqref{linsys4x4} is less than $10^{-10}$, except for matrix 4. }
  \label{fig:WesternUS}
\end{figure}

\begin{figure}[t]   
\centering
  \includegraphics[width=\Width\textwidth]{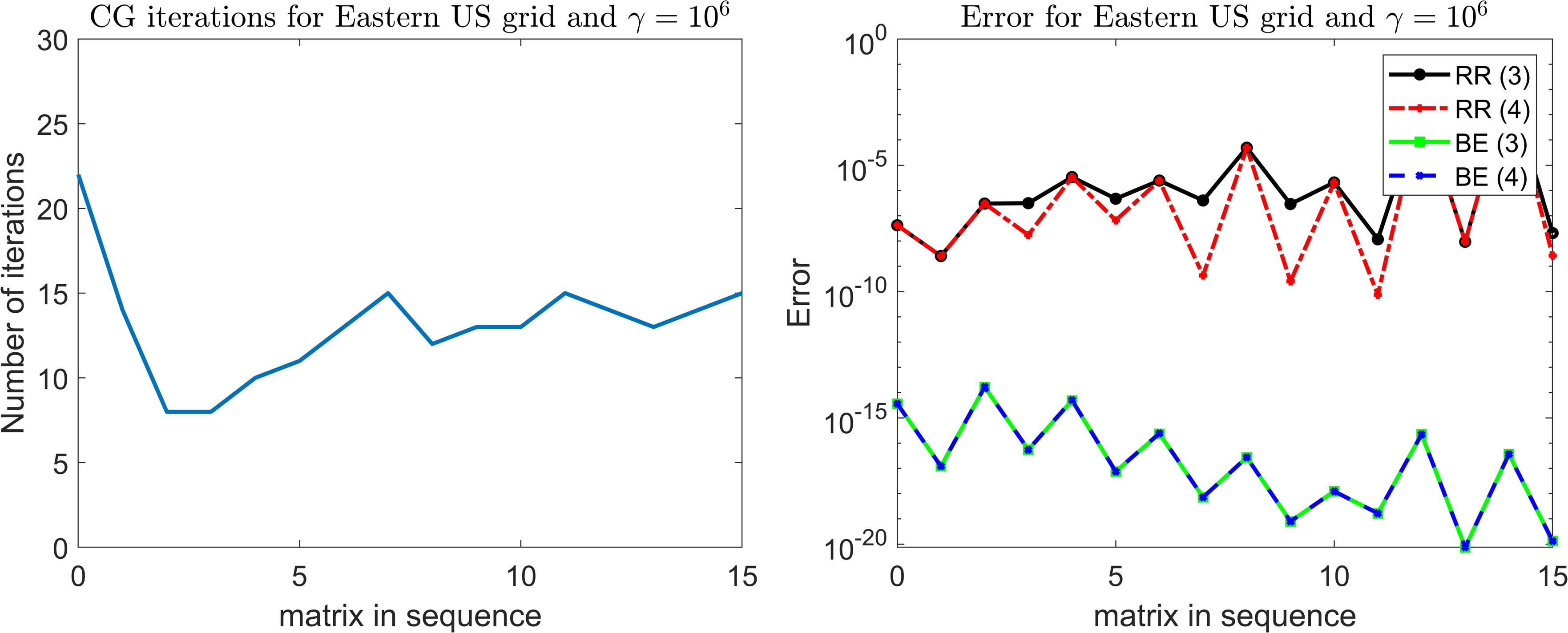}
  \caption{US Eastern Interconnection grid with $\gamma=10^6$ in \eqref{eq:gamma}: (a) CG iterations on \cref{Schur1}. The mean number of iterations is $13.1$. (b) BE and RR for \eqref{linsys4x4} and \eqref{linsys2x2}. The BE for \eqref{linsys4x4} is less than $10^{-10}$. }
  \label{fig:EasternUS}
\end{figure}

Solving larger (and perhaps more poorly conditioned) problems brings about new computational challenges and limits the amount of time any particular task can take. We wish to set $\delta_{\min}$ small enough to avoid over-regularizing the problem, and large enough to eliminate wasteful iterations and numerical issues. We want $\delta_{\max}$ small enough to recognize that we have over-regularized (and should try a different method), but large enough to allow for reasonable regularization. In our numerical tests, we use $\delta_{\min}=10^{-10}$ and $\delta_{\max}$ large enough so that $\delta_1$ can increase until $\Hdel=\Htilde+\delta_1 I$ is SPD.
This informs the parameter selection for the next system.


\Cref{fig:South Carolina grid} shows that the South Carolina grid matrices as currently constructed cannot benefit from this method. They need $\delta_1>1$ to make $\Hgam+\delta_1 I$ SPD, which on a scaled problem means as much weight is given to regularization as to the actual problem. \Cref{alg:CGschur} succeeds on the other matrix sequences, at least for certain $\gamma$s, and needs no regularization ($\delta_1=\delta_2=0$).

For the US Western Interconnection grid, \cref{fig:WesternUS}(a) shows a CG convergence graph and (b) shows several types of error. For the US Eastern Interconnection grid, \cref{fig:EasternUS}(a) shows a CG convergence graph and (b) shows several types of error.  Figures for the Texas grid are given in \cref{appA}, as they do not provide more qualitative insight. Convergence occurs for all matrix sequences in less than $20$ iterations on average. 
The BE for \eqref{linsys4x4} is consistently less than $10^{-8}$ and, with two exceptions in the US Western Interconnection grid, is close to machine precision. Results for the US Eastern Interconnection grid show that the method does not deteriorate with problem size, but rather there are some irregular matrices in the US Western Interconnection grid.

The results in this section suggest that $\delta_{\min}$ in the range $10^{-8}$ down to $10^{-10}$ is reasonable for any $\gamma \leq 10^8$. There is no clear choice for $\delta_{\max}$, but a plausible value would be $\delta_{\max}=2^{10}\delta_{\min}\approx 1000 \delta_{\min}$. This way we are guaranteed that the regularization doesn't take over the problem, and the number of failed factorizations is limited to $10$, which should be negligible in the total solution times for a series of $\approx 100$ problems. 

\subsection{Reordering $\Hgam$}
\label{sec:reorder}

\begin{figure}[h]   
\centering
  \includegraphics[width=\Width\textwidth]{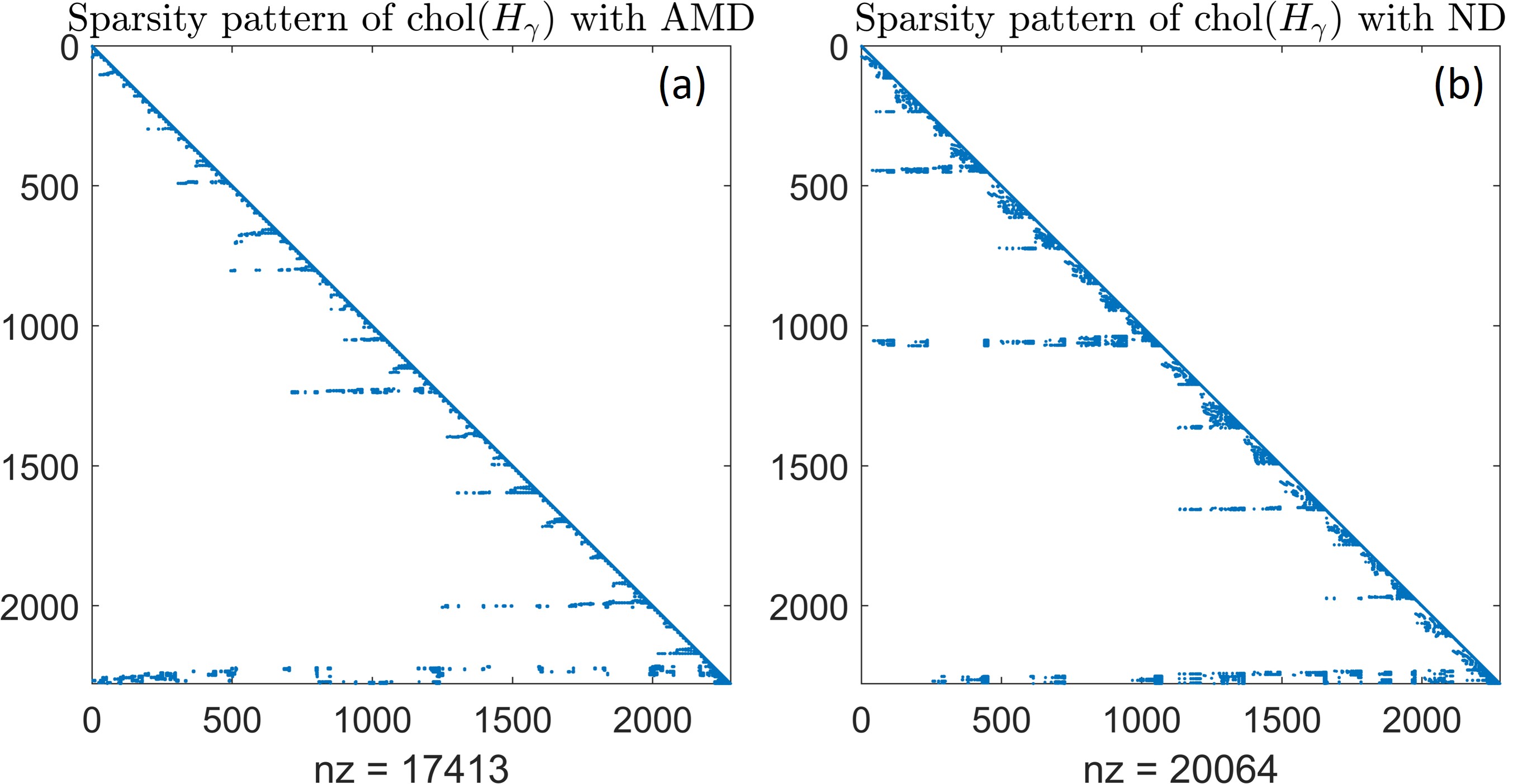}
  \caption{Illinois grid matrix 22: (a) Approximate minimum degree ordering of chol($\Hgam$) is sparser than (b)  Nested dissection ordering of chol($\Hgam$). Both orderings are calculated in \Matlab.}
  \label{fig:cholamd}
\end{figure}

The efficiency of sparse Cholesky factorization $P \Hgam P^T = LL^T$
depends greatly on the row/column ordering defined by permutation $P$.
\Cref{fig:cholamd} compares the sparsity of $L$, corresponding to $\Hgam$ of matrix 22 in the Illinois grid sequence, obtained from two choices
of $P$: approximate minimum degree (\textbf{AMD}) and nested dissection.
(The data in this section are generated via \Matlab.)
We see that AMD produces a sparser $L$ ($17,413$ nonzeros vs.\ $20,064$).

Reverse Cuthill-McKee and no ordering gave $46,527$
and $759,805$ nonzeros respectively. Recall that the sparsity structure is identical for matrices from the same family and similar for other matrix families. As expected, AMD was the sparsest ordering tested for other matrix families. 
Thus, AMD is our reordering of choice.
It is a one-time cost performed during the optimization problem setup. 

\section{Comparison with \LDLT}
\label{sec:compare}

We compare our method with a direct solve of \eqref{linsys4x4} using MA57's \LDLT factorization \cite{Duff2004} with default settings. All testing in this section is done using a prototype \Cpp/CUDA code on a single GPU device. Reference MA57 solutions were computed on a CPU. Further details on computational platforms used are given in~\cref{tab:machines}. 

\begin{table}[t] 
\caption{\label{tab:machines} Accelerator devices and compilers used.  }
\centering
\smallskip
\footnotesize

\begin{tabular}{lllll} \toprule
   Machine name & Host processor & Accelerator device   & Host compiler & Device compiler
\\ \midrule
   Newell       & IBM Power9 & NVIDIA Volta 100     &  GCC 7.4.0       & CUDA 10.2
\\ Deception    & AMD EPYC 7502 & NVIDIA Ampere 100 &  GCC 7.5.0       & CUDA 11.1
\\ \bottomrule
\end{tabular}
\end{table}

The factorization density is $\rho_L=\left(2\,\nnz(L)+\nnz(D)\right)/N$. We define $\rho_C$ analogously for the Cholesky factors of $\Hgam$ in \eqref{linsys2x2mod}, with $\nnz(D)=0$ and dimension $\nx$. Note that $\rho$ gives the average number of nonzeros in the factorization per row or column.
\Cref{tab:ldlcomp} shows that the Cholesky factor of $\Hgam$ is usually less dense than the \LDLT factor for \eqref{linsys4x4}, even though \eqref{linsys4x4} is sparser than \eqref{linsys2x2mod}. 
\Cref{tab:solve} shows the solve times on the Newell computing cluster \cite{PNNLmachines}. The main trend is that as the problem size increases, GPUs using cuSolver~\cite{cuSolverweb} increasingly outperform equivalent computations on a CPU.

\begin{table}[t] 
\caption{\label{tab:ldlcomp} The dimensions, number of nonzeros, and factorization densities (the number of nonzeros in the factors per row) for solving \eqref{linsys4x4} directly with \LDLT ($N$, $\nnz_L$, $\rho_L$ respectively) and for solving \eqref{eq:gamma} with Cholesky ($\nx$, \nnz{}$_C$, $\rho_C$ respectively). Numbers are rounded to 3 digits. K and M signify $10^3$ and $10^6$. For all cases, $\rho_C<\rho_L$ and $\nx<N/2$. }
\centering
\smallskip
\footnotesize

\begin{tabular}{lrrrrrr} \toprule
   Abbreviation        & $N$ & $\nnz{}_L$& $\rho_L$& $\nx$ &$\nnz{}_C$&$\rho_C$
\\ \midrule
Illinois  &  4.64K & 94.7K & 20.4&  2.28K&34.9K &15.3 
\\ Texas &   55.7K & 2.95M&  52.9 & 25.9K&645K &24.9
\\ Western US  &  238K &10.7M&  44.8 & 116K&2.23M& 19.2
\\ Eastern US  &  1.64M & 85.4M& 52.1 & 794K&17.7M& 22.3
\\ \bottomrule
\end{tabular}
\end{table} 

\begin{table}[t] 
  \newcommand{\tenp}[1]{\cdot 10^{#1}}
  \newcommand{\onep}{\cdot 10^0\phantom{-}}
\caption{\label{tab:solve} Average times (in seconds) for solving \eqref{linsys4x4} directly on a CPU with
\LDLT (via MA57~\cite{Duff2004}) or for solving one $\Hdel$ linear system with supernodal Cholesky via Cholmod (CM) in Suite\-Sparse~\cite{Cholmod}, or Cholesky via cuSolver~\cite{cuSolverweb} (CS), each on a CPU and on a GPU. Cholesky on a GPU is quicker than \LDLT on a CPU by an increasingly large ratio. CM GPU does not work for small problems. All runs are on Newell~\cite{PNNLmachines}.}
\centering
\smallskip
\footnotesize

\hspace*{-0pt}%
\begin{tabular}{l@{\qquad}l@{\qquad}l@{\qquad}l@{\qquad}l@{\qquad}l} \toprule
   Name       & MA57            & CM CPU          & CM GPU      & CS CPU          & CS GPU          
\\ \midrule
   Illinois   & $7.35\tenp{-3}$ & $1.74\tenp{-3}$ &             & $2.25\tenp{-3}$ & $5.80\tenp{-3}$    
\\ Texas      & $1.24\tenp{-1}$ & $3.42\tenp{-2}$ &             & $5.67\tenp{-2}$ & $4.79\tenp{-2}$    
\\ Western US & $4.30\tenp{-1}$ & $1.02\tenp{-1}$ &             & $1.89\tenp{-1}$ & $1.59\tenp{-1}$    
\\ Eastern US & $4.34\onep$     & $1.08\onep$     & $3.65\onep$ & $2.52\onep$     & $6.12\tenp{-1}$ 
\\ \bottomrule
\end{tabular}
\end{table}

\begin{table}[t] 
\caption{\label{tab:solvecg} Average times (in seconds) for solving sequences of systems \eqref{linsys4x4} directly on a CPU with \LDLT (via MA57~\cite{Duff2004}) or for solving sequences of systems \eqref{Schur1}--\eqref{Schur2} on a GPU. The latter is split into analysis and factorization phases, and multiple solves. Symbolic analysis is needed only once for the whole sequence. Factorization happens once for each matrix. The solve phase is the total time for Lines 15-17 in \cref{alg:CGschur} with a CG tolerance of $10^{-12}$ on Line 16. The results show that our method, without optimization of the code and kernels, outperforms \LDLT on the largest series (US Eastern Interconnection grid) by a factor of more than $2$ on a single matrix, and more than $3$ on a whole series, because the cost of symbolic analysis can be amortized over the series. All runs are on Deception~\cite{PNNLmachines}.}
\centering
\smallskip
\footnotesize

\hspace*{-0pt}%
\begin{tabular}{lclll} \toprule
   Name     & MA57 &\multicolumn{3}{c}{Hybrid Direct-Iterative Solver}
 \\ \midrule   &  
& Analysis
&Factorization
&Total solves
\\ \midrule
 Illinois	& $6.24\cdot 10^{-3}$	& $3.87\cdot 10^{-3}$ &	$5.07\cdot 10^{-3}$ &$8.55\cdot 10^{-3}$ 
\\ Texas	& $1.00\cdot 10^{-1}$& $2.58\cdot 10^{-2}$	& $3.54\cdot 10^{-2}$ & $1.02\cdot 10^{-1}$	
\\ Western US &	$3.38\cdot 10^{-1}$	& $1.54\cdot 10^{-1}$ &	$1.74\cdot 10^{-1}$ & $1.43\cdot 10^{-1}$
\\ Eastern US &	$3.48\cdot 10^{0\phantom{-}}$ &	$5.81\cdot 10^{-1}$ &	$6.94\cdot 10^{-1}$ &	$3.25\cdot 10^{-1}$
\\ \bottomrule
\end{tabular}
\end{table}
Supernodal Cholesky via Cholmod in Suite\-sparse~\cite{Cholmod} does not perform well on GPUs for these test cases, but performs better than cuSolver on CPUs. 
This matches literature showing that multi\-frontal or supernodal approaches are not suitable for very sparse and irregular systems, where the dense blocks become too small, leading to an unfavorable ratio of communication versus computation~\cite{davis2010algorithm,booth2016basker,he2015gpu}. This issue is exacerbated when supernodal or multifrontal approaches are used for fine-grain parallelization on GPUs~\cite{swirydowicz2020linear}. 
Our method becomes better when the ratio of \LDLT to Cholesky factorization time grows, because factorization is the most costly part of linear solvers and our method has more (but smaller and less costly) system solves.  

\Cref{tab:solvecg} compares a direct solve of \eqref{linsys4x4} using MA57's \LDLT factorization~\cite{Duff2004} and the full CG solve and direct solves on \{\eqref{Schur1}--\eqref{Schur2}\}, broken down into symbolic analysis of $\Hdel$, factorization of $\Hdel$, and CG on \eqref{Schur1} on Deception \cite{PNNLmachines}. 

When Cholesky is used, symbolic analysis is needed only for the first matrix in the sequence because pivoting is not a concern. As problems grow larger, the solve phase becomes a smaller part of the total run time. 
Also, our method increasingly outperforms MA57.  The run time is reduced by a factor of more than $2$ on one matrix from the US Eastern Interconnection grid and more than $3$ on the whole series, because the analysis cost can be amortized over the entire optimization problem.
 This motivates using our method for even larger problems. 
 
 Another advantage of our method is that it solves systems that are less than half the size of the original one, though it does have to solve 
more of them. Notably, the \LDLT factorization may require pivoting during the factorization, whereas Cholesky does not. With MA57, all our test cases required substantial permutations
even with a lax pivot tolerance of 0.01
(and a tolerance as large as 0.5 may be required to keep the factorization stable). This means that for our systems, \LDLT factorization requires considerable communication and presents a major barrier for GPU programming. On the other hand, the direct solve is generally more accurate by 2--3 orders of magnitude.

\section{Iterative vs.~direct solve with $\Hdel$ in \cref{alg:CGschur}}
\label{sec:itdirect}

We may ask if forming $\Hdel \equiv H + D_x + J_d^T D_s J_d + J^T Q J + \delta_1 I$ and its Cholesky factorization $\Hdel = LL^T$ is worthwhile when it could be avoided by iterative solves with $\Hdel$.
Systems \eqref{Schur1}--\eqref{Schur2} require two solves with $\Hdel$ and an iterative solve with $S = J\Hdel\inv J^T + \delta_2 I$, which includes ``inner" solves with $\Hdel$ until CG converges.  As we see in \cref{sec:testing}, 
between 6 and 92 solves with $\Hdel$ are needed in our cases (and possibly more in other cases). Further, the inner iterative solves with $\Hdel$ would degrade the accuracy compared to inner direct solves or would require an excessive number of iterations. Therefore, for iterative solves with $\Hdel$ to be viable, the direct solve would have to cause substantial densification of the problem (i.e., the Cholesky factor $L$ would have to be very dense).
Let $\nnz\op(\Hdel) = \nnz(\Htilde) + 2\,\nnz(J) + \nx$ be the number of multiplications when $\Hdel$ is applied as an operator, and $\nnz\fac(\Hdel)=2\,\nnz(L)$ be the number of multiplications for solving systems with $\Hdel$. These values (generated in \Matlab) and their ratio  are given in \cref{tab:dense}. The ratio is always small and does not grow with problem size, meaning $L$ remains very sparse and the factorization is efficient. As the factorization dominates the total time of a direct solve with multiple right-hand sides, this suggests that performing multiple inner iterative solves is not worthwhile.

\begin{table}[t] 
\caption{\label{tab:dense} Densification of the problem for cases where the direct-iterative method is viable. Numbers are rounded to 3 digits.
K and M signify $10^3$ and $10^6$. $\nnz\op(\Hdel) = \nnz(\Htilde) + 2\,\nnz(J) + n$ is the number of multiplications when $\Hdel$ is applied as an operator, and $\nnz\fac(\Hdel)=2\,\nnz(L)$ is the number of multiplications for solving systems with $\Hdel$.
The ratio $\nnz\fac(\Hdel)/\nnz\op(\Hdel)$
is only about 2 in all cases.}

\centering
\smallskip
\footnotesize

\begin{tabular}{lrrr} \toprule
   Name        & $\nnz\op(\Hdel)$ & $\nnz\fac(\Hdel)$ & ratio
\\ \midrule
Illinois   &  20.5K &  34.9K & 1.70
\\ Texas  &   249K &   646K & 2.59
\\ Western US  &  1.05M &  2.23M & 2.11
\\ Eastern US  &  7.23M &  17.7M & 2.45
\\ \bottomrule
\end{tabular}
\end{table}

\section{Summary}
\label{sec:summary}
 Following the approach of Golub and Greif \cite{GG03}, we have developed a novel direct-iterative method for solving saddle point systems, and  shown that it scales better with problem size than \LDLT on systems arising from optimal power flow~\cite{ACOPF}. The method is tailored for execution on hardware accelerators where pivoting is difficult to implement and degrades solver performance dramatically.
To solve KKT systems of the form \eqref{linsys4x4}, \cref{alg:CGschur} presents a method with an outer iterative solve and inner direct solve. The method assumes
$\Hgam = H + D_x + J_d^T D_s J_d + \gamma J^T J$
is SPD (or almost) for some $\gamma\ge0$, and if necessary,
uses the minimal amount of regularization $\delta_1 \ge 0$
(to within a factor of 2)
to ensure $\Hdel = \Hgam + \delta_1 I$ is SPD. We proved that as $\gamma$ grows large, the condition number of $\Hgam$ grows linearly with $\gamma$, and the eigenvalues of the iteration matrix $S$ converge to $1/\gamma$ (\eqref{eq:S3}). These results provide some heuristics for choosing $\gamma$,
and explain why CG on the Schur complement system \eqref{Schur1} converges rapidly. 
A future direction of research is developing a better method to select $\gamma$.
On several sequences of systems arising from applying an interior method to OPF problems, the number of CG iterations for solving \eqref{Schur1} was less than 20 iterations on average, even though no preconditioning was used. 

Four of the five matrix series were solved with $\delta_1=\delta_2=0$, and the BE for the original system \eqref{linsys4x4} was always less than $10^{-8}$. The efficiency gained by using a Cholesky factorization (instead of \LDLT) and avoiding pivoting is demonstrated in \Cref{tab:ldlcomp}. Even though $\Hgam$ in \eqref{linsys2x2mod} is denser than $K_k$ in \eqref{linsys4x4}, its factors are sparser.
\Cref{tab:solvecg} shows that our method, when it succeeds, has better scalability than \LDLT and is able to utilize GPUs. This is the most substantial result of our paper. For the fifth series (smaller than 2 of the others) $\delta_2=0$ worked, but $\delta_1$ had to be of order 1 and no accurate solution could be obtained.  The development of robust heuristics to select $\delta_1$ and $\delta_2$, and to integrate with the filter line-search algorithm, will be subject of future work.

The fact that the Cholesky factors are scarcely denser than the original matrix suggests that not much could be gained by using nullspace methods~\cite{Rozloznik2018} for the four sequences we were able to solve, as those require sparse LU or QR factorization of $J^T$, which is typically less efficient than sparse Cholesky factorization of $\Hdel$.
For the fifth sequence, and sequences similar to it, an efficient nullspace method may be better than the current fail-safe $\LDLT$ factorization of the $4 \times 4$
system \eqref{linsys4x4}.

\section*{Acknowledgements}

This research was supported by the Exascale Computing Project (ECP), Project Number: 17-SC-20-SC, a collaborative effort of two DOE organizations (the Office of Science and the National Nuclear Security Administration) responsible for the planning and preparation of a capable exascale ecosystem---including software, applications, hardware, advanced system engineering, and early testbed platforms---to support the nation's exascale computing imperative.

We thank Research Computing at Pacific Northwest National Laboratory (PNNL) for computing support. We are also grateful to Christopher Oehmen and Lori Ross O'Neil of PNNL for critical reading of the manuscript and for providing helpful feedback. Finally, we would like to express our gratitude to Stephen Thomas of National Renewable Energy Laboratory for initiating discussion that motivated Section \ref{sec:itdirect}. 


\footnotesize
\bibliography{bibfile,refs-slaven}.
\normalsize
\begin{appendices}
\section{Additional OPF matrix results}
\label{appA}

\begin{figure}[htbp]   
\centering
  \includegraphics[width=\Width\textwidth]{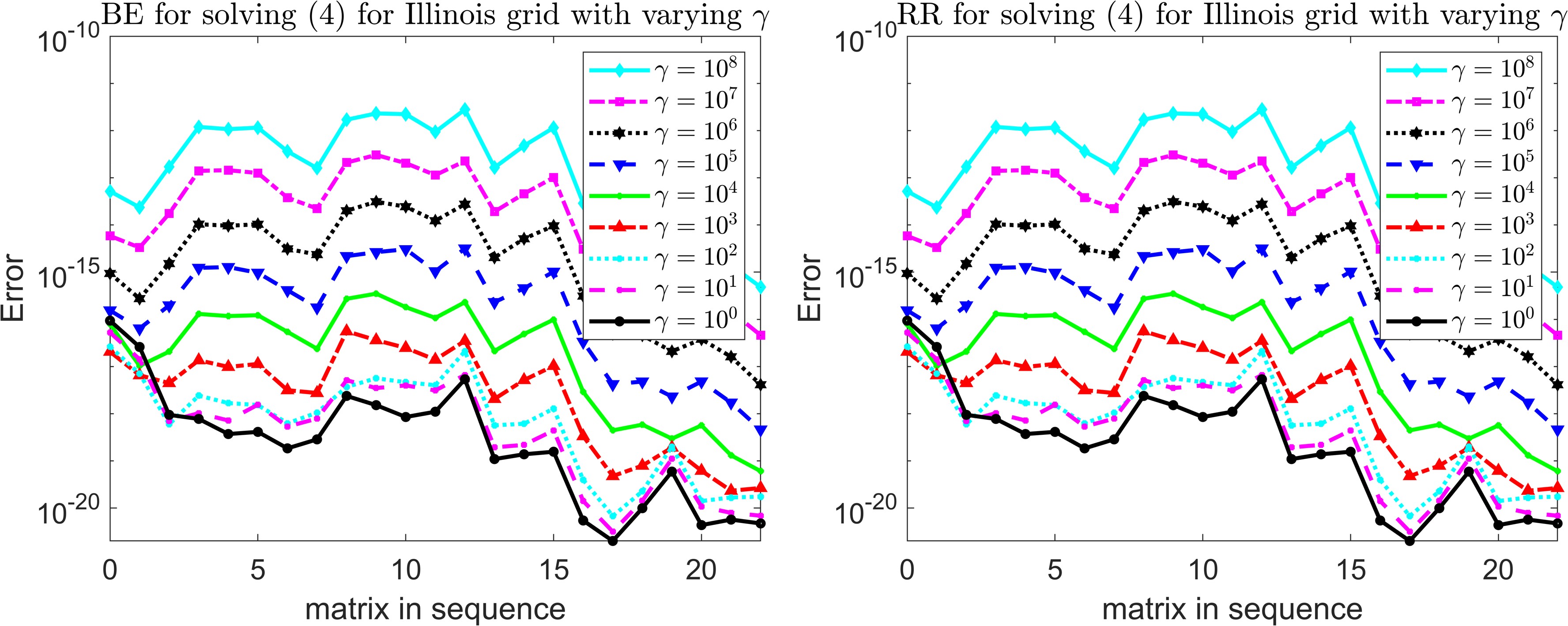}
  \caption{Illinois grid  \eqref{linsys2x2} with varying $\gamma$ in \eqref{eq:gamma}: (a) BE, $\gamma \leq 10^4$ gives results close to machine precision. (b)  RR, $\gamma \leq 10^4$ has $\text{RR} \le 10^{-8}$.}
  \label{fig:alphas200back22low}
\end{figure}

\begin{figure}[htbp]   
\centering
  \includegraphics[width=\Width\textwidth]{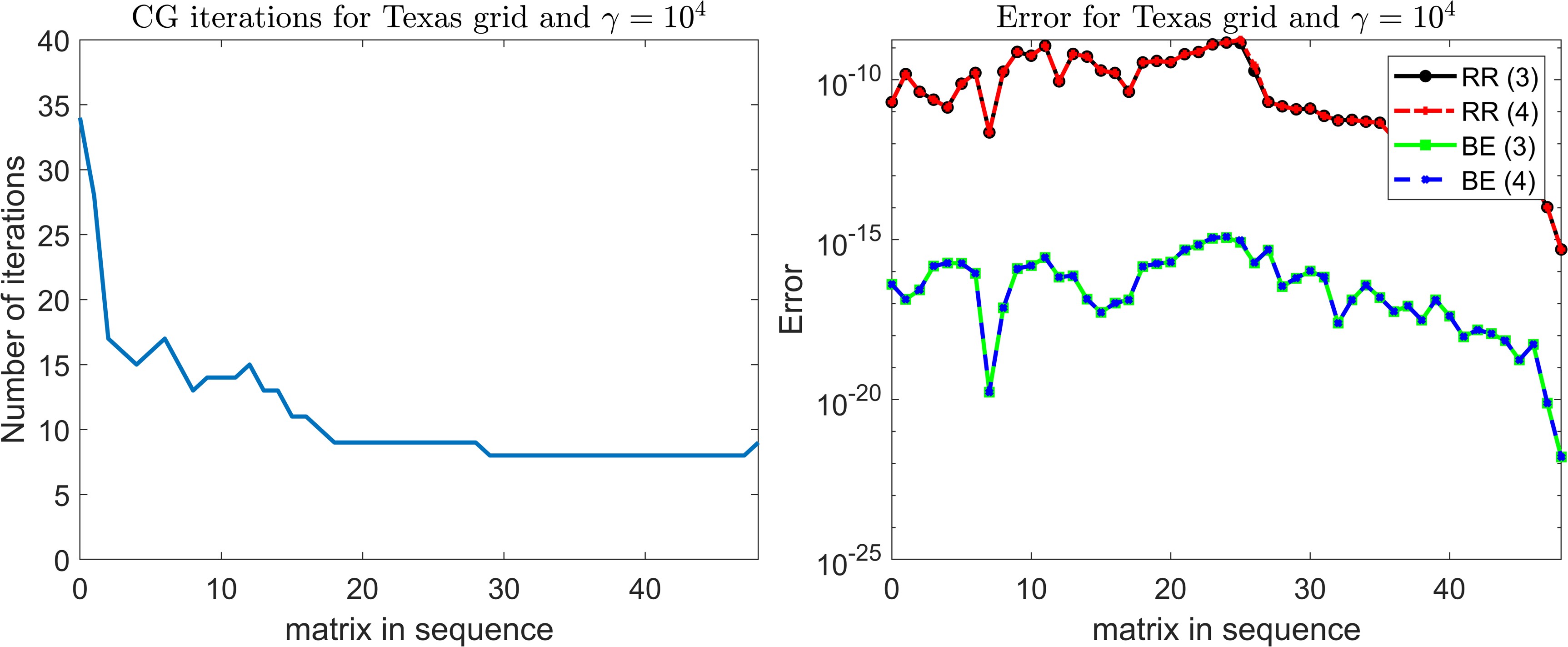}
  \caption{Texas grid with $\gamma=10^4$: (a) CG iterations on \cref{Schur1}. The mean number of iterations is $11.1$. (b) BE and RR for \eqref{linsys4x4} and \eqref{linsys2x2}. The BEs are roughly machine precision and the RRs are less than $10^{-8}$. }
  \label{fig:Texas gridits}
\end{figure}

\Cref{fig:alphas200back22low} shows BE and RR in \eqref{linsys2x2} for varying $\gamma$ for the Illinois grid. For $1\ le \gamma\leq 10^4$ the results are accurate.

CG convergence for the Texas grid with $\gamma=10^4$ is given in \cref{fig:Texas gridits}(a), while (b) shows that the solution is very accurate and all errors are smaller than $10^{-8}$.

 \end{appendices}
\end{document}